\documentclass[12pt]{amsart}
\usepackage{latexsym,fancyhdr,amssymb,color,amsmath,amsthm,graphicx,listings}
\newtheorem{thm}{Theorem}
\newtheorem{lemma}{Lemma}
\newtheorem{propo}{Proposition}
\newtheorem{coro}{Corollary}
\let\paragraph\subsection

\title{On the arboricity of manifolds}
\author{Oliver Knill}
\date{October 22, 2023}
\address{Department of Mathematics \\ Harvard University \\ Cambridge,
MA, 02138 }


\begin{document}

\begin{abstract}
The arboricity of a discrete 2-sphere is always 3. The arboricity of
any other discrete 2-dimensional surface is always 4.
For $d$-manifolds of dimension larger than 2, the
arboricity can be arbitrary large and must be larger than $d$.
\end{abstract}
\maketitle

\section{Introduction}

\paragraph{}
A finite simple graph $G$ is a $d$-manifold if every {\bf unit sphere}
$S(v)$
is $(d-1)$-sphere. It is a {\bf $d$-sphere}, if $G-v$ is contractible
for some
$v$. A graph $G$ is {\bf contractible} if both $S(v)$ and $G-v$ are
contractible. These inductive assumptions start with the empty
graph $0$ being the $(-1)$-sphere and the $1$-point graph $K_1=1$
being contractible.

\paragraph{}
The {\bf arboricity} of graph $G$ is the minimal number of forests
that are needed to partition the edge set of $G$. It is determined
by the maximum of $|E_W|/(|W|-1)$ where $(W,E_W)$ runs over all
induced subgraphs of $G$ with more than one element. The {\bf
Nash-Williams theorem}
assures that ${\rm arb}(G)$ is the smallest integer larger or equal than
this number.

\paragraph{}
For $d$-manifolds, the vertex degree ${\rm deg}(v)$ is larger or equal
than $2d$
because this is the vertex cardinality of the smallest $(d-1)$-sphere.
By the {\bf Euler handshake formula}, $2E/V = \sum_{v} {\rm deg}((v)/V
\geq 2d$, meaning
$E/V \geq d$ and $E/(V-1)>d$ so that the arboricity is at least $d+1$.
The case, where $G$ is a cross polytop, the smallest $d$-sphere, shows
that this
lower bound can indeed happen, at least for spheres.

\paragraph{}
As we will show here, in the cases $d=1$ and $d=2$, we understand
everything: in the case $d=1$, where a manifold a disjoint union of
circular
graphs, we always have arboricity $2$.
In dimension $d=2$, spheres have arboricity $3$ and all other 2
dimensional surfaces
have arboricity $4$. In higher dimensions, the arboricity of $d$-spheres
can take any value in $\{ d+1,d+2, \dots, \}$. For any $d$-manifold
type which is not a sphere, there are examples with arbitrary high
arboricity. We have no information yet about the lower bounds for
manifold
types different from spheres. Already the case of a $d$-torus is open
for $d>2$.
We have seen $3$ tori of arboricity $5$ but no example of arboricity $4$
yet.

\begin{thm}
For a $d$-manifold the arboricity is larger or equal than $d+1$.
For all $d$, there are $d$-spheres with arboricity $d+1$.
For any manifold type and dimension larger than $2$,
there are examples with arbitrary large arboricity. For 2-spheres
the arboricity is $3$ for all other 2-manifold the arboricity is $4$.
\end{thm}

\section{Three forests suffice}

\paragraph{}
The {\bf arboricity} ${\rm arb}(G)$ of a finite simple graph $G=(V,E)$
is the minimal number of forests partitioning the edge set $E$.
A {\bf 2-manifold} is a finite simple graph such that every {\bf unit
sphere} $S(v)$,
the graph induced by all direct neighbors of $v$, is a cyclic graph with
$4$ or
more vertices. We simply call a {\bf 2-manifold} or a {\bf surface}.
For a connected graph, the arboricity is the
same than the minimal number of trees covering $E$.
By the {\bf Nash-Williams theorem} \cite{NashWilliams1964}, ${\rm
arb}(G)$ is the smallest integer
$k$ such that $E' \leq k (V'-1)$ for any induced sub-graph $(V',E')$
with a vertex
subset $V' \subset V$ of more than one element.
 \footnote{We denote the cardinalities of sets with the name
of the set itself.} The {\bf Euler handshake formula} $2E=\sum_{v \in V}
{\rm deg}(v)$ shows
that the arboricity is close to half the maximal average {\bf vertex
degree} ${\rm deg}(v)$
which an induced sub-graph $G'$ of $G$ can have.

\paragraph{}
The {\bf Nash-Williams functional} $\phi(G)=E/(V-1)=\sum_{v \in V} {\rm
deg}(v)/(2V-2)$ is
defined for all graphs with $V \geq 2$ vertices. Again using notation
overload, the
{\bf area} of the graph is denoted by $F$, as is the set of triangles.
The {\bf Euler characteristic} of a $K_4$ graph is defined
as $X=V-E+F$. A $2$-manifold satisfies the {\bf Dehn-Sommerville}
relation
$3F=2E$ because every edge $(a,b)$ is contained in exactly 2 triangles
$abc,abd$, where
$\{c,d\} = S(a) \cap S(b)$ and every triangle contains exactly 3 edges.
For surfaces, the functionals $V,E,F$ and so $\phi,X$ are
spectral properties for the Kirchhoff Laplacian $L$ because
$V={\rm tr}(L^0),E={\rm tr}(L)/2,F={\rm tr}(L)/3,X={\rm tr}(L^0-L^1/6)$.

\begin{lemma}
A surface $G$ of area $F$ and Euler characteristic $X$ has the
Nash-Williams
functional $\phi(G)=3F/(X-1+F)$.
\end{lemma}
\begin{proof}
We use the linear equations $3F=2E$ and $X=V-E+F$ and choose $F$ and $X$
as free variables. The solution is the {\bf $f$-vector}
$(V,E,F)= (X+F/2,3F/2,F)$.
\end{proof}

\paragraph{}
The following result is exercise 21.4.6 in \cite{BM} and some sort of
``folklore".
It generalizes as the {\bf Edmond covering theorem} in matroid theory
\cite{Ruohonen} and has been used as a
tool to show that the star arboricity is less or equal to $6$
\cite{AlgorAlon1989}. We deduced it earlier from having
acyclic chromatic number 4 for all non-prismatic
$2$-spheres and using that for an acyclic coloring with $c$ colors,
all the $c(c-1)/2$ Kempe chains are forests and that if $c$ is even,
one can form groups of $c/2$ chains as one forest producing $c-1$
forests.

\begin{thm}[Algor-Alon]
The arboricity of a planar graph is 3 or less.
\end{thm}

\begin{proof}
Given a planar graph $G=(V,E)$. Because the arboricity of a disjoint
union of graphs
is the maximum of the arboricity of each component, we can assume that
$G$ is connected.
Because removing a leaf (a $v \in G$ with vertex degree ${\rm
deg}(v)=1$)
does not change the arboricity, we
can also assume that $G$ has no leafs. A leafless planar graph can be
triangulated by
adding edges to make it {\bf maximally planar} meaning that adding an
other edge would
no more have it planar. Adding edges without changing the set of
vertices
increases the Nash-Williams functional $\phi(G)=E/(V-1)$. In other
words, $\phi$ as a
function on induced sub-graphs without leafs takes for surfaces its
maximum on $G$.
Every 4-connected component of a maximal planar graph is by a result of
Whitney
\cite{Whitney1931} (see Appendix) either $K_4$ or a triangulation of a
$2$-sphere,
a surface of Euler characteristic 2. The graph $K_4$ can be covered with
3 forests.
That a $2$-sphere can not have arboricity $2$ follows from the fact that
$3F=2E, V-E+F=2$ implies $V-1=1+F/2$ and so
$$   \phi(G)=E/(V-1)=(2E)/(2F)=3F/(2+F) \; . $$
Setting the ratio $\phi(G)$ to $2$ gives $3F=4+2F$, so that $F \leq 4$,
which is
an impossibility for a $2$-sphere. Setting the ratio $\phi(G)=3$ gives
$F/(2+F)=1$
which is not possible for any $F$.
\end{proof}

\paragraph{}
This proof shows that any subgraph $G$ of a $2$-manifold
for which all leafs have been recursively removed, has the property that
any polygon face (a locally minimal polygon) can be completed. This then
leads to a 2-manifold. The arboricity is smaller or equal than the
arboricity
of the full triangulation. This means that if we have a global bound
$\phi(G) \leq c$ for all manifolds of a specific topological type, then
$\phi(G) \leq c$ for all sub-graphs of such manifolds. More generally,
for
any surface $\phi(H)$ as a functional on this class of sub-graphs takes
the
{\bf global maximum} on $G$. We have never seen an example, where an
induced subgraph $H$ of $G$ has $\phi(H)>\phi(G)$.

\paragraph{}
Let us look a bit more at the functional $\phi(G)=E/(V-1)$. The
Nash-Williams theorem states that the maximum over all $\phi(W,E_W)$
with induced subgraphs $(W,E_W)$ determines the arboricity. For
$2$-manifolds, already the single number $\phi(G)$ determines the
arboricity.
Lets us formulate this and also note that we do not know whether this
holds in higher
dimensions. The use of {\bf critical graphs}, minimal graphs for which
the arboricity
is larger than expected has been used already by Nash-Williams.

\begin{coro}[$\phi$ on $G$ determines]
a) All 2-spheres have arboricity $3$ if and only if $2 \leq \phi(G)<3$
   for all 2-spheres. \\
b) All 2-manifold different from spheres have arboricity $4$ if and only
   if $3 \leq \phi(G)<4$ for all such 2-manifolds.
\end{coro}

\begin{proof}
a) As we have seen, that arboricity $3$ implies $2 \leq E/(V-1)<3$ is
a direct consequence of the Nash-Williams theorem.
That checking $E/(V-1) <3$ is enough, follows from the fact that we
can look at the smallest subgraph $G_W=(W,E_W)$ generated by $W$ for
which
$E_W/(W-1) \geq 3$ and so arboricity being equal to $4$. This critical
$G$
can not have leaves as otherwise, we could take away a leaf and
have no minimality. Now we have polygonal faces.
Triangulate those by adding edges.
This can only increase the ratio. But then we end up with a
2-manifold or $K_4$.  \\
b) Take the 2-manifold $G$ of smallest genus for which we have failure
and $\phi(G)=4$.  Then, take the smallest induced subgraph $H_W$ for
which the Nash-Williams bound fails. This graph must be leafless.
Now add edges until we have a triangulation.
That triangulation however either is a 2-manifold  with $E/(V-1) \geq 4$
or a manifold of smaller degree with $E/(V-1) \geq 4$
which was excluded by the minimality on the degree.
\end{proof} .

\paragraph{}
We derived the planarity result first
from a result about acyclic chromatic number of 2-spheres which relies
on
the 4-color theorem.
We learned that from Chat GPT who did not provide reference that the
result must
be known. A literature search then led to the exercise in \cite{BM}.
{\bf ``Google Bard"} gave us, when pressed, as reference
\cite{NashWilliams1964}
which however does not mention this result on planar graphs.
Antropic's ``Claude" gave the following proof:
{\it removing an edge reduces the number of edges and faces by one
(which is not true for $K_2$)
and claims that removing $E - (V-2)$ edges to get a forest leading to $E
\leq 3(V - 2)$
(also not true for $K_2$).} So far, we have not found a reference which
states and
proves the result in detail and so can label it as a ``folklore result".
\cite{AlgorAlon1989} states the $E \leq 3(V - 2)$ bound too, which
however fails for
$K_2$. The Algor-Alon article however is the first we could locate,
where the result
is stated and proven in a way that a reader can fill the details.

\section{Surfaces}

\paragraph{}
A finite simple graph for which every unit sphere $S(v)$ is a circular
graph with $4$ or more elements, is a {\bf two-dimensional discrete
manifold}.
We simply call this also a {\bf surface}. If $(a,b)$ is an edge, then
$S(a) \cap S(b)$
is a unit sphere in a cyclic graph and so a zero-dimensional sphere (a
graph with 2
vertices $(c,d)$ and no edge). There are now two triangles $(abc)$ and
$(abd)$.

\paragraph{}
The Euler characteristic of a surface is $X=V-E+F$, where $F$ is the
number of {\bf triangles},
$K_3$ sub-graphs of $G$. A connected surface of Euler characteristic is
2 is called
a {\bf 2-sphere}. The class of 2-spheres together with $K_4$ are known
by Whitney
to agree with the class of maximal planar graphs that are 4-connected
(see Appendix).
The graph $K_4$ is the only planar graph which is 4-connected and
contains $K_4$.
Any connected surface which is planar and has Euler characteristic $2$,
must be a sphere.

\paragraph{}
The main result in this note is:

\begin{thm}[Surface arboricity]
The $2$-sphere has arboricity 3. All other surfaces have arboricity $4$.
\end{thm}

\paragraph{}
In general, if $G$ is a graph and $H$ is a sub-graph, then the
arboricity
of $H$ is smaller or equal than the arboricity of $G$
because a sub-graph of a forest is again a forest. The obvious
monotonicity
also follows from the Nash-Williams formula. In general however, the
Nash-Williams
ratio for an induced sub-graph $H$ could be larger for a sub-graph. For
example, if
we remove a single isolated vertex, then the ratio goes up.
The number $E/(V-1)$ alone does not determine the
arboricity of $G=(V,E)$ in general, even so we suspect that it does so
for manifolds.

\paragraph{}
Let $F$ denote the set of triangles.
\footnote{Again write the cardinalities of $V,E,F$ without the absolute
signs. }
The Euler characteristic $X=V-E+F=2$ and the Dehn-Sommerville equation
$3F=2E$
show that the {\bf $f$-vector} $(f_0,f_1,f_2) = (V,E,F)$ encoding the
cardinalities of
all simplices in $G$ is determined by the {\bf area} $F$ alone:

\begin{lemma}
For a 2-manifold with Euler characteristic $X$, we have
$(V,E,F)  = (X+\frac{F}{2},\frac{3F}{2},F)$.
\end{lemma}
\begin{proof}
The second equation is the Dehn-Sommerville $2E=3F$, the first is the
Euler gem
formula.
\end{proof}

\begin{coro}[Algor-Alon]
If $G$ is a 2-sphere, then $E<3V-6$ and so $E/(V-1)<3$.
\end{coro}
\begin{proof}
From the lemma, we see that $3F/2 < 2+F/2-6$ which holds for $F >4$
But every 2-sphere satisfies this inequality.
\end{proof}

\begin{coro}[Lower bounds]
Let $G$ be a connected 2-manifold of Euler characteristic $X$. \\
a) If $X=2$, then the arboricity is at least 3. \\
b) If $X \leq 1$, the arboricity is at least 4.
\end{coro}
\begin{proof}
a) The above lemma gives for the entire graph the Nash-Williams
ratio $E/(V-1) = 3F/(2+F)$. This is $\geq 2$ for $F \geq 4$ and
so by Nash-Williams, we have arboricity larger than 2. \\
b) For $X=1$, which forces a connected surface to be a projective plane,
we have $E/(V-1)=3$ so that the arboricity is already 4.
For $X\leq 0$, we have $E/(V-1)>3$.
\end{proof}

\section{Upper bound}

\paragraph{}
In order to get upper bound for the arboricity, we will need to estimate
how many faces $F$ are needed for a given Euler characteristic $X$.
If $G$ is a surface $G$, define $K(x) = 1-{\rm deg}(x)/6$ as the
{\bf curvature} of $G$ at the point $x$.

\begin{lemma}[Gauss Bonnet]
$\sum_{x \in V} K(x) = X$.
\end{lemma}
\begin{proof}
The Euler characteristic is $\sum_{x \in V \cup E \cup F} \omega(x)$,
where
$\omega(x)=(-1)^{{\rm dim}(x)}$. Leave the value $\omega(x)=1$ on $v$
for $v \in V$,
move the value $\omega(x)=-1$
from each $x \in E$ to the two vertices and each value $\omega(x) = 1$
for $x \in F$ to
its three vertices. This leads to $K(v) = 1-{\rm deg}(v)/2 + {\rm
deg}_2(v)/3$,
where ${\rm deg}_2(v)$ are the number
of triangles containing $v$. In the manifold case, ${\rm deg}_2(v)={\rm
deg}(v)$
so that $K(v) = 1-{\rm deg}(v) (1/2-1/3) = 1-{\rm deg}(v)/6$.
\end{proof}

\paragraph{}
It is an interesting challenge to find the minimal area for a 2-manifold
of a certain type.
We do not need to know the minimum but can find some lower bounds. We
think that $a),b),c)$
are sharp and that c) is the sharp discrete Loewner's inequality:

\begin{coro}
Assume $G$ is a connected 2-manifold of Euler characteristic $X$. \\
a) If $X=2$, then $G$ has at least $8$ faces. This is the minimum. \\
b) If $X=1$, then $G$ has at least $14$ faces \\
c) If $X=0$, then $G$ has at least $20$ faces. \\
d) If $X<0$, then $G$ has at least $12-14X$ faces.
\end{coro}

\begin{proof}
We have $(V,E,F)= (X+F/2,3F/2,F)$ and $V=X+F/2$ gives
$$    F = 2(V-X)  \; . $$
a) The curvature can not be larger than $1/3$, so that
   we need at least $2*3=6$ vertices which gives $F=2(V-X)=8$.  \\
b) Since we need to have a closed loop changing the orientation of a
triangle, we
need at least $7$ vertices. In order to make this a manifold, we need at
least
one vertex more so that $V \geq 8$.  This implies $F=2(V-X)=2(8-1)=14$.
\\
c) Look at the two shortest generators of the fundamental group.
   These are two curves of length $4$ which intersect a point $a$. This
   gives 7 points.
   By Gauss-Bonnet two of the points must have at least 6 neighborhood
   this gives 3 new
   points at least so that $V \geq 10$. This implies $F \geq
   2(V-X)=2V=20$. \\
   Given any triangulation of a 2-torus, we can get a triangulation of a
   Klein
   Bottle and vice versa by reversing the identification along one of
   the loops,
   we again have $F \geq 20$. \\
d) Since $X=2-b_1$ or $X=1-b_1$, depending on orientability,
   we have $b_1=2-X$ or $b_1=1-X$. There must be at least
   $3*2*b_1=12-6X$ or $6-6X$ vertices. With $V \geq 6-6X$ we have
   $F=2(V-X) \geq 2(6-7X) \geq 12-14X$ vertices.
\end{proof}

\paragraph{}
All these estimates could certainly be improved:
the {\bf systole} of a discrete 2-torus $G$ is the length of the
shortest
non-contractible curve in $G$. It is $4$.
A discrete version of {\bf Loewners inequality} will give
$4^2 \leq F/2$. This gives $F \geq 32$.

\paragraph{}
The next proposition shows that the {\bf projective plane} is the
threshold for
arboricity. We should point out that with ``discrete projective plane"
we do of course
not mean finite projective planes over finite fields but finite graphs
which are 2-manifolds
and have Euler characteristic 1. In any example, the geometric
realization
(which we do not look at) would be a classical 2-dimensional projective
plane.

\begin{propo}
a) If $X=2$, then $E/(V-1) < 3$. \\
b) If $X=1$, then $E/(V-1) = 3$. \\
c) If $X \leq 0$, then $E/(V-1) <4$ meaning $F > 8-8X$.
\end{propo}
\begin{proof}
We have $V=X+F/2, E=3F/2$. For $F \to \infty$ the ration $E/(V-1)$
converges to 3.
We need to show that in the case $X=2$ it is always smaller than 3 and
that in the
case $X \leq 1$, it is always smaller than $4$. This produces lower
bounds for the
area $F$. \\
a)   $E/(V-1)=3F/(2+F) <3$. This is always the case. We do not need any
conditions.
b)   In the case $X=1$ we have $E/(V-1)=3$. The Nash-Williams functional
is independent of
     the surface! We can not get lower nor higher. The arboricity is
     $4$.  \\
c)   Lets look at the equation $E/(V-1)=3F/(2X+F-2) = 4$.
     This gives us an upper bound for $F$ for which we possibly could
     have
     Nash-Williams ratio $4$. This means $3F = 8X+4F-8$. This means
     $F=8-8X$.
\end{proof}

\begin{coro}
There is no surface of arboricity $5$.
\end{coro}
\begin{proof}
(i)  For $X=2$, we have seen the arboricity is $3$.  \\
(ii) For $X=1$, we have the arboricity at least $4$ because of the
previous
proposition b). \\
(iii) For $X=0$ we need at least 28 faces. \\
(iv)  Lets look at $X<0$. But there $F \leq 8-8X$ is
incompatible with $F \geq 10-14X$.
\end{proof}

\section{Remarks}

\paragraph{}
The {\bf Barycentric refinement} of a graph $G$ takes the complete
subgraphs as the
new vertex set and connects two if one is contained in the other.

\begin{propo}
a) For $X=2$, the 2-spheres, the Barycentric refinement increases the
ratio $E/(V-1)$.
b) For $X=1$, the projective plane the refinement leaves the ratio
invariant.
c) For $X \leq 0$, the refinement reduces the ratio.
\end{propo}
For given Euler characteristic $X$, the area $F$ alone determines both
$V$ and $E$ and
so the Nash-Williams fraction. We have seen $\phi(G) = E/(V-1) =
3F/(X-2+F)$
We have $V'=V+E+F=X-2E=X+3F, E'=2E+6F=9F$, $F'=6F$.
Therefore $E'/(V'-1) = 9F/(X-1+3F)$.
It follows that in the Barycentric limit, $E/(V-1)$ converges to $3$.

\paragraph{}
Given any surface type. For which surface do we maximize or minimize
the Nash-Williams functional $E/(V-1)$ on sub-graphs.
Is there an example of a surface for which $\phi((V,E))= E/(V-1)$
evaluated on a subgraph generated by a subset of the vertex set
is smaller than the functional evaluated on the manifold itself?

{\bf Conjecture:} for any $d$-manifold (a graph for which every unit
sphere $S(v)$
is a $(d-1)$-sphere), the Nash-Williams functional has its global
maximum at $G$.
This fails for non-manifolds like  $G=K_2+K_1$ of $H=K_2$ and $K_1$ for
example
where $\phi(K_2+K_1) = 1/3$ but $\phi(H)=1/2$ which is larger than
$\phi(G)$.
For surfaces, if $\phi(G)$ determines the arboricity, it would be a
spectral property.

\paragraph{}
For $d$-manifolds, graphs for which every unit sphere $S(v)$ is a
$(d-1)$-sphere,
with $d \geq 3$ the situation is much different:

\begin{thm}
For any d-manifold type for $d \geq 3$, there are examples of arbitrary
large arboricity.
\end{thm}
\begin{proof}
It is enough to look at a $3$-manifold because for a $d>3$ manifold,
we can look at intersections of unit spheres which are 3-manifolds.
Assume we want to reach the target arboricity $n$. First refine
edges to render an edge degree larger than $n$. We call this graph $H$.
Note that this
can increase the vertex degrees also. Now, refine the edge $m$ times.
Each such refinement
adds one vertex and $n+1$ edges. So, after
$m$ steps, (always doing it with the same edge) and $m$ large enough, 
we have the fraction $E/V$
(for the graph $H$ with at least one large vertex degree larger than
$n$) change to
$$ \frac{E}{V} \to \frac{E + m(n+1)}{V+m} \; . $$
In the limit $m \to \infty$, this gets closer and closer to $n+1$. We
will
therefore eventually be larger than $n$. By the Nash-Williams theorem,
also the arboricity
of that refined graph will be larger than $n$.
\end{proof}

\paragraph{}
The arboricity of a $d$-sphere can take values $d+1$ or more higher.
We first believed that there is no $d$-manifold $G$ with arboricity
larger than $c_d$,
the Perron-Frobenius bound given in the Barycentric limit.
See \cite{ThreeTreeTheorem}.

\paragraph{}
Here are two questions. \\
{\bf 1)} Is the arboricity of a graph a spectral property for
graphs or at least for $d$-manifolds? The later would be true
if the above stated conjecture would hold. \\
{\bf 2)} The minimal possible arboricity of a $d$-manifold type is a
topological invariant by definition. Can one use this to distinguish
manifolds? The minimum for $d$-spheres is $d+1$. What is the minimum
for $d$-tori?

\section*{Appendix: the theorem of Whitney of 1931}

\paragraph{}
The argument of reducing the 4-color problem to planar graphs which are
triangulations
goes back to Arthur Cayley \cite{Crilly}. Hassler Whitney
\cite{Whitney1931} proves a
variant of the lemma below (we used it in \cite{knillgraphcoloring}).
A graph $G=(V,E)$ is {\bf maximally planar} if one can not increase the
edge set $E$ (while keeping the vertex set $V$ fixed) without losing the
planar property.
Planarity is by Kuratowski a finite combinatorial notion not needing the
continuum:
the class of planar graphs are the graphs which do not contain any edge
refined
version of $K_5$ nor $K_{3,3}$.

\paragraph{}
A graph $G=(V,E)$ is called {\bf $4$-connected}, if removing $3$ or less
vertices keeps
the graph connected. We always assume $G$ to be connected = 1-connected.
We make the choice to call the empty graph $0$ not connected and $K_1$
1-connected by not
2-connected and $K_2$ 2-connected by not 3-connected and $K_3$
3-connected but not
4-connected and $K_4$ 4-connected but not $5$-connected.
Every $2$-sphere is $4$-connected, as we will see. A wheel graph
with boundary $C_n$ with $n \geq 4$ is never 4-connected as we can
separate any of the boundary
points $v$ by removing its unit sphere $S(v)$ which is a path graph with
$3$ vertices.
Graphs with more than 4 elements for which some vertex degree ${\rm
deg}(v) \leq 3$,
is not 4-connected. The octahedron graph is a 2-sphere. It is $4$
connected but not $5$-connected.

\paragraph{}
The next statement is not in the literature as such, because the notion
of a
2-sphere is usually defined differently, like using notions of
triangulations
of Euclidean spheres. The statement on page 389 of the paper
\cite{Whitney1931}
comes close. The argument of reducing the 4-color problem to planar
graphs which are triangulations goes back to Arthur Cayley and is part
of ``folklore", meaning
often referred to without bothering to reference or prove it. According
to \cite{Stromquist}
(work from 1975 having the bad luck to be in the shadow of the computer
assisted proof
\cite{AppelHaken1} from 1976), almost everybody who has thought about
the
4-color theorem has been led to ideas like reducing to 4-connected and
maximally planar graphs.
We first believe that coloring 2-spheres could be easier than coloring
planar graphs
as there are more tools like the Sard theorem \cite{KnillSard},
telling that for a locally injective function $f$ on a $d$-manifold, the
level curves
are all $(d-1)$-spheres or empty. Barycentric refinement changes
chromatology however.
The level curves are in the Barycentric refinement $G_1$ of $G$
a graph $G$ of maximal dimension $2$, the chromatic number of $G_1$ is
$3$
as $f(x)={\rm dim}(x)$ is a color. For 2-manifolds already, the
chromatic number can be 5
and is conjectured to be $5$ or less by a {\bf conjecture of Albertson
and Stromquist} of 1982
\cite{AlbertsonStromquist} (using slightly different definitions for
manifolds).

\begin{lemma}[Whitney]
A graph with at least one edge is both $4$-connected and
maximal planar, if and only if is a $2$-sphere or $K_4$.
\label{whitneylemma}
\end{lemma}
\begin{proof}
\fbox{{\bf (i)} $\Leftarrow$} First assume that $G$ is $K_4$ or a
$2$-sphere. \\
{\bf a)} If $G=K_4$, then it is maximally planar and 4-connected:
The maximal planarity is clear since $K_4$ is planar (Kuratowski
theorem) and
because all edges are drawn.
To see 4-connectedness, note that removing 3 vertices keeps a single
vertex, which is by definition a
connected graph. Assume from now on in part (i) of this proof that $G$
is a 2-sphere: \\
{\bf b)} $G$ must be {\bf tetrahedral-free} because $S(v)$ can not
contain $K_3$ by definition of $2$-spheres. \\
{\bf c)} To see that $G$ is $4$-connected, we use contraposition and
assume that it is now.
There are then 3 vertices $a,b,c$ which when removed, split the graph
into two disconnected parts $A,B$.
Case (1) If these three vertices are not all connected to each other,
then it is a path
graph $abc$ of length $2$ in the sphere which cuts the sphere into two
parts.
But this means $a$ and $c$ are boundary points.
Case (2) If the three vertices form a triangle then necessarily there is
a path around the
triangle. This is the union of all balls $B(a),B(b),B(c)$ minus $a,b,c$.
But this implies that removing the three vertices does not separate the
graph.
Any previous path connecting two points $x,y$ which passes through $a$
or $b$ or $c$
can be rerouted to take the detour. \\
{\bf d)} $G$ is planar: as $G$ is a 2-sphere, removing one vertex in $G$
produces
a {\bf contractible} graph $H$, meaning recursively that it is either
$K_1$ or that there
exists a vertex $v$ for which both $S(v)$ or $G \setminus v$ are both
contractible.
$H$ is now a graph with interior or boundary: the interior has circular
graphs $S(v)$,
the boundary has path graphs $S(v)$. That such a graph is planar can be
seen by
by induction, starting the smallest of this kind, the wheel graphs which
do not contain
any possibly refined $K_{3,3}$ nor $K_5$. Every extension step which
adds a vertex and keeps
the property of being a ball can be drawn in the plane and at every
stage the
boundary of the graph is a cyclic graph. Each extension step is a
pyramid extension
over a path subgraph of the boundary. \\
{\bf e)} $G$ is maximal planar: adding an other edge $(x,y)$ with $y$
different
from $S(x)$ is not possible because the edge would have to cross the
circle $S(x)$
which by the Jordan curve theorem can not happen.
[ The reference to the Jordan curve theorem uses the classical Jordan
curve theorem.
Alternatively, if $x,y$ are connected in $S(x)$ intersected with $S(y)$
consists
of exactly two points $a,b$ which are not adjacent. The unit sphere
$S(a)$ is
not connected as it contains part of $S(x)$ and part of $S(y)$ which are
not connected.] \\

\fbox{{\bf (ii)} $\Rightarrow$:} Now assume that $G$ is a 4-connected,
maximally planar graph $G$
different from $K_4$. We want to show that it is a 2-sphere. By
assumption, $G$ has at least one edge.
Define a {\bf face} in $G$ is a closed path in the graph which encloses
a
connectivity component. $A \subset \mathbb{R}^2$ of the complement of a
planar
embedding $\tilde{G}$ of $G$. [Alternatively if we do not want to refer
to the continuum,
it can be defined as a minimal non-contractible loop in $G$. This loop
can be also a triangle.]
If there is no face, then there is no closed loop in the graph the graph
is
a tree, which is not $4$-connected. So there is at least one face.  \\
{\bf a)} Every face of $G$ is a triangle: this follows from maximal
4-connected planarity:
   if we had an $n$-gon face with $n>3$,
   we could add additional diagonal connections, without violating
   either 4-connectivity nor planarity. \\
{\bf b)} The graph $G$ can not contain $K_4$: this would contradict
$4$-connectivity
     or imply that $G$ itself is $K_4$. \\
{\bf c)} A unit sphere $S(x)$ is triangle-free: otherwise we had a
tetrahedral unit ball $B(x)$
     contradicting b). \\
{\bf d)} The unit sphere $S(x)$ is connected: if it were disconnected,
then
         an additional connection in $S(x)$ could be added. This would
         violate the maximal
         4-connected planarity. \\
{\bf e)} The vertex degree ${\rm deg}(y)$ of every $y \in S(x)$ within
$S(x)$
       is larger than $1$: assume $y \in S(x)$ was a leaf, having only
       one neighbor $a$
       in $S(x)$.
       As removing both $x,y$ keeps the graph connected, $a$ is
       connected via a
       path to an other point $b \in S(x)$ and so directly connected to
       a neighboring
       point $y \in S(x)$. \\
{\bf f)} The vertex degree of $y$ in $S(x)$ can not be $3$: removing
$x,y,b$ with not
        interconnected $(a,b,c)$ neighboring $y \in S(x)$ would render
        $G$ disconnected
        as a path $(a,\dots ,h, \dots,c)$ would lead to a homeomorphic
        copy of $K_{3,3}$
        containing vertices $(h,y,x,a,b,c)$ inside $G$ contradicting the
        Kuratowski theorem. \\
{\bf g)} Each unit sphere is cyclic: from steps c)-f) follows that
$S(x)$ is cyclic with
        $n \geq 4$. The reason is that we have a graph for which every
        vertex $y$ in
        $S(x)$ has exactly $2$ disconnected neighbors. This means that
        $S(x)$ is
        a cyclic graph $C_n$ with $n \geq 4$. \\
{\bf h)} If we remove one vertex, the graph $H=G-v$ is contractible: we
can assume
         that $G$ is larger than $K_4$. The graph has interior points,
         the vertices which were not connected to $v$ and so have
         circular unit spheres as well as boundary points, vertices
         which were connected to $v$.
         As we have seen already that all unit spheres $S(v)$ in $G$
         were circular graphs, the set of boundary points of $H$
         forms a circular graph. Take a boundary point $w$ away. The set
         of boundary points still forms a circular graph.
         Now add again $v$ and connect to the new boundary. This
         procedure is technically
         an edge collapse. We might have generated a $K_4$ subgraph. But
         then we either have $K_4$ or can remove
         the vertex with $S(v)$ being a triangle. We have made the graph
         smaller and can see by induction
         that $H$ was contractible.
\end{proof}

\section*{Appendix: about higher dimensional manifolds}

\paragraph{}
For every $k$-simplex $x=(x_0,\dots,x_{k})$ in a $d$-manifold (meaning a
$K_k$ subgraph),
the intersection of all unit spheres $S(x_j)$ is a $(d-k-1)$-dimensional
sphere.
For $k=d-2$, it is 1-sphere. Its length is the {\bf degree} ${\rm
deg}(x)$ of
that simplex $x$. For $d=2$, this is the usual vertex degree.
The $f$-vector $(f_0,\dots,f_d)$ counting the sub-simplices is in small
dimension
renamed with letters like $(V,E,F)$ for surfaces.
The quantities $V,E,F,C,H,\dots$  counting the number of vertices,
edges, faces, chambers, hyper-chambers etc are all {\bf valuations}
meaning
that $X(A \cup B) = X(A) + X(B) - X(A \cap B)$. The edge degrees ${\rm
deg}_k(x)$ in
dimension $k$ divided by $k+1$ are the {\bf curvatures} of $f_k$.
Formulas like $\sum_{x \in F} {\rm deg}(x)= 4C$ are {\bf Gauss-Bonnet
formulas}
for these valuations. Here are some lower dimensional cases: \\

\begin{small}
\begin{tabular}{|l|l|l|l|l|} \hline
dim  & f-vector      &   Dehn-Sommerville  & degree                     
& X \\ \hline
1    & $(V,E)$       &   $2E=2V$           &                           
& $V-E$ \\
2    & $(V,E,F)$     &   $3F=2E$           & $\sum_{x \in V} {\rm
deg}(x) = 2E$ & $V-E+F$ \\
3    & $(V,E,F,C)$   &   $4C=2F$, $X=0$    & $\sum_{x \in E} {\rm
deg}(x) = 3F$ & $V-E+F-C$ \\
4    & $(V,E,F,C,H)$ &   $5H=2C$,$22E+40C=33F+45H$  & $\sum_{x \in F}
{\rm deg}(x) = 4C$ & $V-E+F-C+H$ \\ \hline
\end{tabular}
\end{small}

The relation $22E+40C=33F+45H$ is the only surprising one up to
dimension $4$.
See \cite{KnillBarycentric,KnillBarycentric2}.

\paragraph{}
A {\bf 3-manifold} is a finite simple graph $G$ for which every unit
sphere $S(v)$
is a 2-sphere. A {\bf 2-sphere} $S$ is a graph for which every unit
sphere is a 1-sphere,
a circular graph with 4 or more elements such that $S-w$ is contractible
for any $w$.
The f-vector $(f_0,f_1,f_2,f_3)$ of a 3-manifold is also denoted
$(V,E,F,C)$, where
$V$ is the number of vertices, $E$ the number of edges, $F$ the number
of triangular
faces and $C$ the number of chambers, tetrahedral sub-graphs. The
f-function
$f(t)=1+\sum_k f_k t^{k+1}$ for a 3-manifold is the quartic polynomial
$1+Vt+Et^2+Ft^3+Ct^4$.

\paragraph{}
The smallest 3-manifold is the 3-sphere $S=C_4 \oplus C_4$, where
$C_4=S_0 \oplus S_0$
is the cyclic graph with 4 vertices and $H \oplus G$ is the {Zykov join}
obtained by
taking two disjoint copies $H + G$ of $H$ and $G$, then connecting all
vertices of $H$
with all vertices of $G$. The join of a $k$-sphere and an $l$-sphere is
in general a
$k+l+1$-sphere. The $f$-functions of joins multiply $f_{H \oplus G}(t) =
f_H(t) f_G(t)$.
So, $f_{S}(t) =f_{S_4}(t)^2 = (1+4t+4t^2)^2 =  1+8t+24 t^2+32 t^3 +16
t^4$.
The {\bf Euler characteristic} $X=V-E+F-C$ can also be written as
$1-f_G(-1)$.

\paragraph{}
The Euler characteristic of any 3-manifold is zero. One can see this
from the
{\bf general Gauss-Bonnet formula}
$$ f'_G(t) = \sum_{v \in V} f_{S(v)}(t) \; . $$
See \cite{DehnSommerville,dehnsommervillegaussbonnet}.
For a $3$-manifold, where every unit sphere is a $2$-sphere, we have
$f_{S(v)}(-1)=-1$ and $f_{S(v)}$ is odd around the center $t=-1/2$, then
$f_G'(t)$ is odd around the center $t=-1/2$ and so $f_G(t)=1+\int_0^t
f'(s) \; ds$ is
even around the center $t=-1/2$ which means $f_G(-1) = f_G(0)=1$ so that
the Euler
characteristic is 0.

\paragraph{}
There are various other ways to see that an odd dimensional manifold has
zero
Euler characteristic. The most recent one is \cite{Sphereformula} for
simplicial
complexes. A very early approach is via Poincar\'e-Hopf
\cite{poincarehopf},
leading in the manifold case to an index formula \cite{indexformula}
based
on index expectation \cite{indexexpectation}.  We tried in the past to
use
index expectation to understand better the Hopf conjectures
\cite{DiscreteHopf2,DiscreteHopf} stating that for even dimensional
manifolds,
positive sectional curvature implies positive Euler characteristic.

\paragraph{}
The  Dehn-Sommerville relations for 3-manifolds is $4C=2F$.
It complements the other Dehn-Sommerville
relation $V-E+F-C=0$. Of interest is the {\bf average vertex degree}
$2E/V$ and the
{\bf average edge degree} $6C/E = 3F/E$. The product of the average
vertex degree and edge degree is
$12 C/V$ which is an {\bf average chamber density}. In the Barycentric
limit, the quantities
$E/V,C/E$ or $C/V$ become {\bf universal}. They can be read off from the
Perron-Frobenius
eigenvector of the Barycentric refinement operator $A$ (the eigenvector
to the largest
eigenvalues $24$). See \cite{ThreeTreeTheorem}.

\paragraph{}
What happens for $3$-manifolds or $3$-dimensional {\bf Dehn-Sommerville
spaces}
(manifold like spaces where each unit sphere is a Dehn-Sommerville
graph),
is that the arboricity number $E/V$ determines all
quantities $(V,E,F,C)$ up to scale:

\begin{propo}
For 3-manifolds $C/V=E/V-1$ and $F/V = 2E/V-2$.
\end{propo}
\begin{proof}
Given $E$ and $V$, we have $E-V=F-C$ and since
$F=2C$ we have $C=E-V$ and $F=2E-2V$.
\end{proof}

This means $E/V=1+C/V$ so that
$1+C/V$ gives a lower bound on arboricity.

\paragraph{}
There is a $3$-torus which has $E/V=7$ so that the arboricity can
already be $8$
for $3$-tori. For 3-spheres, we can construct graphs of arboricity $4$.
An example with arboricity $4$ is the cross polytop with $(8, 24, 32,
16)$.

\paragraph{}
A $4$-manifold is a finite simple graph such that every unit sphere
$S(v)$ is a
$3$-sphere. Now there are two interesting Dehn-Sommerville invariants
for the
$f$-vector $(V,E,F,C,H)$, where $H$ is the number of {\bf
hyper-chambers}, $4$-dimensional
units of space. We have $5H=2C$, the usual one and the more interesting
identity
\begin{equation}
\label{Dehn-Sommerville}
 -22E+33F-40C+45H=0 \; .
\end{equation}
They are derived from 2 eigenvectors of $A^T$, where $A$
is the Barycentric refinement operator
$\left[ \begin{array}{ccccc} 1 & 1 & 1 & 1 & 1 \\ 0 & 2 & 6 & 14 & 30 \\
0 & 0 & 6 & 36 & 150 \\
                   0 & 0 & 0 & 24 & 240 \\ 0 & 0 & 0 & 0 & 120 \\
                   \end{array} \right]$.

\paragraph{}
The curvature of the $5$-manifold at a vertex $v$ is
$$ K(v) =
1-\frac{V(v)}{2}+\frac{E}{3}-\frac{F}{4}+\frac{C}{5}-\frac{H}{6} \; . $$
The equations $H=2C/5,E=(-40C+33F+45H)/22, V-E+F-C+H=2$
simplify this to $K(v)=0$ for all $v$. We see from this computation that
the {\bf zero curvature condition}
$K(x)=0$ is equivalent to the nontrivial Dehn-Sommerville equation.
When writing \cite{cherngaussbonnet}, we had known only experimentally
that for odd dimensional
manifolds, the curvature is always zero. Both index expectation and
Dehn-Sommerville then
confirmed this picture.

\bibliographystyle{plain}

\end{document}